\documentclass[a4paper,12pt,centertags]{amsart}
\usepackage[utf8]{inputenc}
\usepackage{mathrsfs}
\usepackage{eulervm}

\usepackage[T1]{fontenc}
\usepackage[a4paper,asymmetric]{geometry}
\usepackage{tikz}
\usepackage[colorlinks,pdfdisplaydoctitle]{hyperref}
\usepackage{mathscinet}
\newtheorem*{maintheorem}{Theorem}
\newtheorem{theorem}{Theorem}[section]

\newtheorem{lemma}[theorem]{Lemma}
\newtheorem{proposition}[theorem]{Proposition}
\newtheorem{corollary}[theorem]{Corollary}
\theoremstyle{definition}
\newtheorem{definition}[theorem]{Definition}
\newtheorem{assumption}[theorem]{Assumption}
\theoremstyle{remark}
\newtheorem{remark}[theorem]{Remark}
\numberwithin{equation}{section}

\DeclareMathAlphabet{\mathbbm}{U}{bbm}{m}{n}
\newcommand{\uno}{\mathbbm{1}}
\DeclareMathOperator{\Div}{div}
\newcommand{\e}{\operatorname{e}}
\newcommand{\im}{\mathrm{i}}
\newcommand{\N}{\mathbf{N}}
\newcommand{\Z}{\mathbf{Z}}
\newcommand{\R}{\mathbf{R}}
\newcommand{\loc}{\textrm{loc}}
\newcommand{\scal}[2][{}]{\langle #2 \rangle_{#1}}
\newcommand{\Pb}{\mathbb{P}}
\newcommand{\E}{\mathbb{E}}
\newcommand{\Ons}{\Omega_\text{NS}}
\newcommand{\field}[1][F]{\mathscr{#1}}
\newcommand{\Torus}{\mathbb{T}_3}
\newcommand{\Energy}{\mathcal{E}}
\newcommand{\Test}{\mathcal{D}^\infty}
\newcommand{\semi}{\mathcal{P}}
\newcommand{\cov}{\mathcal{Q}}
\newcommand{\de}[1][t]{\,\mathrm{d}#1}
\newcommand{\term}[1]{\text{\textcircled{\sf\scriptsize #1}}}
\newcommand{\property}[1]{\textsf{\footnotesize\bfseries[#1]}}
\newcommand{\nuno}{\property{n1}}
\newcommand{\ndue}{\property{n2}}
\newcommand{\besp}[1]{\text{\tiny$(#1)$}}
\newcommand{\tu}{\widetilde{u}}
\newcommand{\vettore}[1]{{\mathbf{#1}}}
\newcommand{\vk}{\vettore{k}}
\newcommand{\vl}{\vettore{l}}
\newcommand{\vm}{\vettore{m}}
\hypersetup{%
  pdftitle={Critical strong Feller regularity for Markov solutions to the Navier--Stokes equations},
  pdfsubject={The main purpose of this paper is to show that Markov solutions
    to the 3D Navier--Stokes equations driven by Gaussian noise have the strong
    Feller property up to the critical topology given by the domain of the Stokes
    operator to the power one-fourth.},
  pdfauthor={M. Romito},
  pdfkeywords={stochastic Navier-Stokes equations, martingale problem, Markov
    selections, strong Feller property, continuous dependence, ergodicity},
  pdfcreator={M. Romito}
}
\begin{document}
  \title[critical strong Feller for Markov solutions to Navier--Stokes]{Critical strong Feller regularity for Markov solutions to the Navier--Stokes equations}
  \author[M. Romito]{Marco Romito}
    \address{Dipartimento di Matematica, Universit\`a di Firenze\\ Viale Morgagni 67/a\\ I-50134 Firenze, Italia}
    \email{romito@math.unifi.it}
    \urladdr{\url{http://www.math.unifi.it/users/romito}}
  \begin{abstract}
    The main purpose of this paper is to show that Markov solutions to the 3D
    Navier--Stokes equations driven by Gaussian noise have the strong Feller
    property up to the critical topology given by the domain of the Stokes
    operator to the power one-fourth.
  \end{abstract}
\maketitle
\section{Introduction}

It is not known whether the martingale problem for the Navier--Stokes equations
driven by Gaussian noise is well--posed~\cite{DapDeb08,Rom08a}. In order to
analyse the problem Da~Prato and Debussche~\cite{DapDeb03} (see also
\cite{DebOda06,Oda07}) showed the existence of Markov processes solutions to
the equations and some regularity properties of the transitions semigroups.

A different approach to the existence and regularity of Markov solutions
has been introduced in~\cite{FlaRom06,FlaRom08} (see also
\cite{FlaRom07,Rom08,Rom08a,Rom08b,BloFlaRom09,RomXu09,GolRocZha09}), based on
an abstract selection principle for Markov families (see Theorem~\ref{t:markov})
and the short time coupling with a smooth process. A refined analysis of this
coupling is one of the purposes of this paper (see Sections~\ref{s:sf}
and~\ref{ss:approx}).

Here we consider the Navier--Stokes equations on the three dimensional
torus $\Torus$ with periodic boundary conditions,
\begin{equation}\label{e:nse}
  \begin{cases}
     \dot u + (u\cdot\nabla)u + \nabla p = \nu\Delta u + \dot\eta,\\
     \Div u = 0.
  \end{cases}
\end{equation}
driven by a Gaussian noise. For simplicity we can represent the noise as
\[
\dot\eta = \sum_{\vk\in\Z^3}\sigma_\vk\de[\beta]_\vk(t)\e^{\im\vk\cdot x},
\]
where $(\beta_\vk)_{\vk\in\Z^3}$ are (suitably) independent Brownian motions
(precise definitions and assumptions will be given in the next section).
The analysis originated in \cite{FlaRom08} used in a crucial way two main
assumptions on the driving noise, namely regularity and non-degeneracy.
The property of non-degeneracy can be translated, roughly speaking, in terms
of the coefficients $(\sigma_\vk)_{\vk\in\Z}$ simply as $\sigma_\vk>0$.
The possibility to relax this condition is analysed in Romito and
Xu~\cite{RomXu09}.

The main purpose of this paper is to complete the analysis developed in
\cite{FlaRom06,FlaRom07,FlaRom08,Rom08,Rom08a,Rom08b} and relax the regularity
assumption, namely to allow coefficients whose decay as $|\vk|\to\infty$ is of order
$|\sigma_\vk|\approx |\vk|^{-3/2-2\alpha_0}$ for $\alpha_0>0$. In~\cite{FlaRom08}
the restriction is $\alpha_0>\tfrac16$, so the improvement seems tiny. On the
other hand the following result achieved here is, in a way, the best possible.
\begin{maintheorem}
Assume non-degeneracy (as explained above) and let $\alpha_0>0$. Then every
Markov solution to the Navier-Stokes equations is strong Feller in the topology
of $D(A^\alpha)$ for every $\alpha>\tfrac12$, where $A$ is the Stokes operator.
\end{maintheorem}
This optimality has a twofold reason. On one hand, the value of the main parameter
$\alpha_0\leq0$ would correspond to \emph{non-trace class} covariance and the
analysis of the Navier-Stokes equations in this case is open. On the other hand
the main theorem above states that under this assumption \emph{every} solution
has good regularity properties \emph{as long as} the underlying equation admits
local smooth solutions. In fact, the value $\tfrac12$ is the critical threshold
for existence and uniqueness of smooth solution in the deterministic case,
as proved by Fujita and Kato~\cite{FujKat64}. An explanation of the critical
value, of the connection with the scaling properties of the equation and in
general of the scaling heuristic for the Navier--Stokes equations can be found
for instance in Cannone~\cite{Can04}.

In conclusion in this paper we verify that every Markov diffusion generated
by the Navier--Stokes equations has good properties of regularity as long as
it lives in the largest possible space (at least in the hierarchy of hilbertian
Sobolev spaces) dictated by the deterministic analysis.

The paper is organised as follows. Section~\ref{s:generic} contains notations
and a short summary of those definitions and result useful for this work. The
strong Feller property in strong topologies is proved in Section~\ref{s:sf}. The
main theorem (recast as Theorem~\ref{t:main}) is proved in Section~\ref{s:critical}
and some additional properties of the Markov solutions which follow from it
are given in Section~\ref{ss:consequence}. Finally in Section~\ref{s:technical}
we prove some technical results: the construction of the short time coupling
with a smooth solutions and an inequality for the Navier--Stokes nonlinearity.
\section{Generalities and past results}\label{s:generic}

Let $\Torus=[0,2\pi]^3$ and let $\Test$ be the space of infinitely
differentiable divergence free periodic vector fields with mean zero on $\Torus$.
Let $H$ be the closure of $\Test$ in $L^2(\Torus,\R^3)$ and $V$ be
the closure in $H^1(\Torus,\R^3)$. Denote by $A$, with domain $D(A)$, the
\emph{Stokes} operator and for every $\alpha\in\R$ set $V_\alpha=D(A^{\alpha/2})$,
with norm $\|u\|_\alpha = \|A^{\alpha/2}u\|_H$ for $u\in V_\alpha$. In particular
we have $V_0=H$, $V_1=V$ and $V_{-1}=V'$. Define the bi-linear operator
$B:V\times V\to V'$ as the projection onto $H$ of the nonlinearity $(u\cdot\nabla)u$
of equation~\eqref{e:nse}. We refer to Temam \cite{Tem95} for a detailed account
of all the definitions.

We recast problem~\eqref{e:nse} in the following abstract form,
\begin{equation}\label{e:absnse}
  du + (\nu Au + B(u,u))\de = \cov^{\frac12}\de[W],
\end{equation}
where $W$ is a cylindrical Wiener process on $H$ and $\cov$ is a linear
bounded symmetric positive operator on $H$ with finite trace.

The probabilistic framework for problem~\eqref{e:absnse} is given as follows.
Set $\Ons = C([0,\infty);D(A)')$, let $\field[B]$ be the Borel $\sigma$-field
on $\Ons$ and let $\xi:\Ons\to D(A)'$ be the canonical process on $\Ons$ (that
is, $\xi_t(\omega)=\omega(t)$). Define the filtration
$\field[B]_t=\sigma(\xi_s:0\leq s\leq t)$.

We give the definition of solutions following the approach presented
in~\cite{Rom08b}, which we briefly recall. For every $\varphi\in\Test$ consider
the process $(M_t^\varphi)_{t\geq0}$ on $\Ons$ defined for $t\geq0$ as
\[
  M_t^\varphi
   = \scal[H]{\xi_t-\xi_0,\varphi}
     + \nu\int_0^t\scal[H]{\xi_s,A\varphi}\de[s]
     - \int_0^t\scal[H]{B(\xi_s,\varphi),\xi_s}\de[s].
\]
\begin{definition}
  Given $\mu\in\Pr(H)$, a probability $\Pb_\mu$ on $(\Ons,\field[B])$ with
  marginal $\mu$ at time $t=0$ is a \emph{weak martingale solution} starting at
  $\mu$ to problem~\eqref{e:absnse} if
  \begin{itemize}
    \item $\Pb_\mu[L^2_\loc([0,\infty);H)]=1$,
    \item for each $\varphi\in\Test$ the process $(M_t^\varphi,\field[B]_t,\Pb_\mu)$
      is a square integrable continuous martingale with quadratic variation
      $[M^\varphi]_t=t\|\cov^{\frac12}\varphi\|^2_H$.
  \end{itemize}
\end{definition}
Let $(\sigma_k^2)_{k\in\N}$ be the system of eigenvectors of the covariance $\cov$
and let $(e_k)_{k\in\N}$ be a corresponding complete orthonormal system of
eigenfunctions. Define for every $k\in\N$ the process $\beta_k(t)=\sigma_k^{-1}M_t^{e_k}$.
Under a weak martingale solution $\Pb$, $(\beta_k)_{k\in\N}$ is a sequence of
independent one dimensional Brownian motions, thus the process
\begin{equation}\label{e:wiener}
  W(t) = \sum_{k=0}^\infty\sigma_k\beta_k(t)e_k
\end{equation}
is a $\cov$-Wiener process and $z(t) = W(t) - \nu\int_0^t A\e^{-\nu A(t-s)}W(s)\de[s]$
is the associated Ornstein-Uhlenbeck process starting at $0$, that is the solution
to
\begin{equation}\label{e:stokes}
  dz + \nu Az\de = \cov^{\frac12}\de[W],
  \qquad z(0)=0.
\end{equation}
Define the process $v(t,\cdot) = \xi_t(\cdot) - z(t,\cdot)$. Since
$M_t^\varphi=\langle W(t),\varphi\rangle$ for every test function $\varphi$,
it follows that $v$ is a weak solution of the equation
\begin{equation}\label{e:Veq}
  \partial_t v + \nu A v + B(v+z,v+z) = 0, \qquad\Pb-\text{a.~s.},
\end{equation}
with initial condition $v(0)=\xi_0$. The energy balance functional associated
to $v$ is given as
\[
  \Energy_t(v,z)
    = \frac12\|v_t\|_H^2 + \nu\int_0^t\|v_r\|_V^2\de[r] - \int_0^t\scal{z_r,B(v_r+z_r,v_r)}\de[r].
\]
\begin{definition}\label{d:ems}
  Given $\mu\in\Pr(H)$, a weak martingale solution $\Pb_\mu$ starting at $\mu$
  is a \emph{energy martingale solution} if
  \begin{itemize}
    \item $\Pb_\mu[v\in L_\loc^\infty([0,\infty);H)\cap L^2_\loc([0,\infty);V)]=1$,
    \item there is a set $T_{\Pb_\mu}\subset (0,\infty)$ of null Lebesgue measure
      such that for all $s\not\in T_{\Pb_\mu}$ and all $t\geq s$,
      $\Pb_\mu[\Energy_t(v,z)\leq\Energy_s(v,z)] = 1$.
  \end{itemize}
\end{definition}
The following theorem ensures existence of a Markov family of solutions to
problem~\eqref{e:nse}.
\begin{theorem}[\cite{Rom08b}]\label{t:markov}
There exists a family $(\Pb_x)_{x\in H}$ of energy martingale solutions
such that $\Pb_x[\xi_0=x]=1$ for every $x\in H$ and for almost every $s\geq0$
(including $s=0$), for all $t\geq s$ and all bounded measurable $\phi:H\to\R$,
\[
  \E^{\Pb_x}[\phi(\xi_t')|\field[B]_s] = \E^{\Pb_{\xi_s}}[\phi(\xi_{t-s}')].
\]
\end{theorem}
In the rest of the paper, we shall consider the following assumption on the
covariance operator.
\begin{assumption}\label{a:mainass}
The covariance operator $\cov$ of the driving noise satisfies
  \begin{itemize}
    \item[\nuno] there is $\alpha_0>0$ such that $A^{\frac34+\alpha_0}\cov^\frac12$
      is a linear bounded operator on $H$,
    \item[\ndue] $A^{\frac34+\alpha_0}\cov^\frac12$ is a linear bounded invertible
      operator on $H$, with bounded inverse.
  \end{itemize}
\end{assumption}
We shall emphasize when we need the stronger property {\ndue} or, vice versa,
when the weaker property {\nuno} is sufficient for our purposes.
\section{The strong Feller property}\label{s:sf}

In this section we extend \cite[Theorem~5.11]{FlaRom08} and \cite[Theorem~3.1]{FlaRom07}
to all the admissible values of $\alpha$ and $\alpha_0$ where a short time
coupling with smooth solutions is possible (see Theorem~\ref{t:Rexistuniq}).
\begin{definition}
A semigroup $(\semi_t)_{t\geq0}$ is $V_\alpha$--\emph{strong Feller}
at time $t>0$ if $\semi_t\varphi\in C_b(V_\alpha)$ for every
$\varphi:H\to\R$ bounded measurable.
\end{definition}
\begin{theorem}\label{t:sf}
Under Assumption~\ref{a:mainass}, let $\alpha>\tfrac12$ be such that
\[
\max\{1 + \alpha_0, \frac12 + 2\alpha_0\} \leq \alpha < 1 + 2\alpha_0
\]
(with $\alpha>\max\{1 + \alpha_0, \tfrac12 + 2\alpha_0\}$ if $\alpha_0=\tfrac12$).
Then the transition semigroup $(\semi_t)_{t\geq0}$ associated to any Markov
solution $(\Pb_x)_{x\in H}$ is $V_\alpha$--strong Feller for every $t>0$.
Moreover, there are $c>0$ and $\gamma\geq2$ (whose value is given in the proof)
such that for all $\phi\in B_b(H)$, $x\in V_\alpha$ and $h\in V_\alpha$ with
$\|h\|_\alpha\leq 1$,
\begin{equation}\label{e:lip}
  |\semi_t\phi(x+h) - \semi_t\phi(x)|
    \leq \frac{c}{t\wedge 1}(1+\|x\|_\alpha^\gamma)\|h\|_\alpha\log\bigl(\e\|h\|_\alpha^{-1}\bigr).
\end{equation}
\end{theorem}
\begin{proof}
We follow the lines of the proof of \cite[Theorem 5.11]{FlaRom08}. Let
$x\in V_\alpha$ and $h\in V_\alpha$ with $\|h\|_\alpha\leq 1$, and choose
$R\geq 3(1+\|x\|_\alpha)$. Fix $t>0$ and let $\epsilon>0$ be such that
$\epsilon\leq c R^{-\gamma}$ (where $c$ $\gamma$ are so that
Proposition~\ref{p:blowup} holds true) and
$\epsilon\not\in T_{\Pb_x}\cup T_{\Pb_{x+h}}$, where $T_\Pb$ is the set
of exceptional times where the energy inequality fails to hold for $\Pb$ (see
Definition~\ref{d:ems}). Then for every $\phi\in B_b(H)$ with $\|\phi\|_\infty\leq 1$,
\begin{multline*}
  |\semi_t\phi(x+h) - \semi_t\phi(x)|
    \leq |\semi_\epsilon\psi_\epsilon(x+h) - \semi_\epsilon^\besp{\alpha,R}\psi_\epsilon(x+h)| + {}\\
        + |\semi_\epsilon^\besp{\alpha,R}\psi_\epsilon(x+h) - \semi_\epsilon^\besp{\alpha,R}\psi_\epsilon(x)|
        + |\semi_\epsilon^\besp{\alpha,R}\psi_\epsilon(x) - \semi_\epsilon\psi_\epsilon(x)|,
\end{multline*}
where we have set $\psi_\epsilon = \semi_{t-\epsilon}\phi$ and we have used the
Markov property (in the version of Theorem~\ref{t:markov}). Now, by
Theorem~\ref{t:weakstrong} and Proposition~\ref{p:blowup},
\[
\begin{aligned}
  |\semi_\epsilon^\besp{\alpha,R}\psi_\epsilon(x) - \semi_\epsilon\psi_\epsilon(x)|
    & = \E^{\Pb_x^\besp{\alpha,R}}[\psi_\epsilon(\xi_\epsilon)\uno_{\{\tau_x^\besp{\alpha,R}<\epsilon\}}]
      - \E^{\Pb_x}[\psi_\epsilon(\xi_\epsilon)\uno_{\{\tau_x^\besp{\alpha,R}<\epsilon\}}]\\
    &\leq c\|\phi\|_\infty\e^{-c\frac{R^2}{\epsilon}},
\end{aligned}
\]
and similarly for the term in $x+h$. The middle term can be estimated using
either Propositions~\ref{p:RSFhigh} or~\ref{p:RSFlow}, depending on the value
of $\alpha$. We consider first the case $\alpha>\tfrac32$, so that
\[
  |\semi_t\phi(x+h) - \semi_t\phi(x)|
    \leq  c_1\e^{-c_2\frac{R^2}{\epsilon}}
        + \frac{c_1}{\epsilon}\|h\|_\alpha\e^{c_3R^2\epsilon}
\]
for constants $c_1,\dots,c_3$ and $R\geq 3(1+\|x\|_\alpha)$,
$\epsilon\leq (c_4R^{-2})$ and $\epsilon\leq \tfrac12(t\wedge 1)$. As in the proof of
\cite[Theorem 3.1]{FlaRom07}, we choose the values
$R=3(1+\|x\|_\alpha)$ and $\epsilon\approx(1\wedge t\wedge c_4R^{-2})/(-\log(\|h\|_\alpha/\e))$
to get~\eqref{e:lip}.

On the other hand, if $\alpha\leq\frac32$, then
\[
  |\semi_t\phi(x+h) - \semi_t\phi(x)|
    \leq  c_1\e^{-c_2\frac{R^2}{\epsilon}}
        + \frac{c_1}{\epsilon}\|h\|_\alpha\e^{c_3R^\gamma\epsilon}
\]
for $R\geq 3(1+\|x\|_\alpha)$, $\epsilon\leq (c_4R^{-\gamma})$ and
$\epsilon\leq \tfrac12(t\wedge 1)$, with $\gamma=4/(3+4\alpha_0-2\alpha)$.
A similar choice of $\epsilon$ and $R$ leads again to~\eqref{e:lip}.
\end{proof}
The rest of the section contains the arguments needed to complete the proof
of the above theorem.
\subsection{Differentiability of the approximated flow}

Given $\alpha\in(\tfrac12,1+2\alpha_0)$,
let $\semi_t^\besp{\alpha,R}\varphi(x) = \E[\varphi(u_x^\besp{\alpha,R}(t)]$
be the transition semigroup associated to problem~\eqref{e:strongR}, with
$x\in V_\alpha$ and $\varphi:H\to\R$ bounded measurable. In this section
we analyse the regularity of this semigroup.
\begin{proposition}\label{p:RSFhigh}
Assume {\nuno} and {\ndue} of Assumption~\ref{a:mainass}. Given $R\geq1$ and $\alpha$
such that
\begin{equation}\label{e:AlfaCond2}
  \alpha>\frac32
  \quad\text{and}\quad
  \frac12 + 2\alpha_0 \leq \alpha < 1+2\alpha_0,
\end{equation}
the transition semigroup $(\semi_t^\besp{\alpha,R})_{t\geq0}$ associated to
problem \eqref{e:strongR} is $V_\alpha$-strong Feller for all $t>0$. Moreover,
there are numbers $c_1>0$ and $c_2>0$ such that for every $x_0\in V_\alpha$,
for every $\varphi\in B_b(H)$ and for every $h\in V_\alpha$,
\[
  |\semi^\besp{\alpha,R}_t\varphi(x_0+h)-\semi^\besp{\alpha,R}_t\varphi(x_0)|
    \leq \frac{c_1}{t\sqrt{\nu}}\|h\|_\alpha\e^{\frac{c_2}{\nu}R^2t}\|\varphi\|_\infty.
\]
\end{proposition}
\begin{proof}
Fix $\alpha$ as in \eqref{e:AlfaCond2} and let $t>0$ and $\varphi\in B_b(H)$
with $\|\varphi\|_\infty\leq1$. We proceed as in~\cite[Proposition 5.13]{FlaRom08}.
By the Bismut, Elworthy and Li formula,
\begin{equation}\label{e:BEL}
  \begin{split}
    |\semi^\besp{\alpha,R}_t\varphi(x_0+h)-\semi^\besp{\alpha,R}_t\varphi(x_0)|
     &\leq \frac{c}{t}\sup_{\eta\in[0,1]}\E^{P^\besp{R}_{x_0+\eta h}}\Bigl[\bigl(\int_0^t\|\cov^{-\frac12}D_h\xi_s\|_H^2\de[s]\bigr)^{\frac12}\Bigr]\\
     &\leq \frac{c}{t}\sup_{\eta\in[0,1]}\E^{P^\besp{R}_{x_0+\eta h}}\Bigl[\bigl(\int_0^t\|D_h\xi_s\|_{\frac32+2\alpha_0}^2\de[s]\bigr)^{\frac12}\Bigr],
  \end{split}
\end{equation}
since $\|\cov^{-\frac12}D_h\xi_s\|_H\leq C\|D_h\xi_s\|_{3/2+2\alpha_0}$,
by {\ndue} on $\cov$, and so we only have to estimate the inner integral.
For every $x\in V_\alpha$ and $h\in V_\alpha$,
denote by $u_x^\besp{R}$ the process solution to \eqref{e:strongR}
starting at $x$, and by $\tu = D_h u^\besp{R}_x$ the derivative
of the flow in the direction $h$. Then $\tu$ solves
\begin{multline}\label{e:flowderivative}
  \partial_t\tu + \nu A\tu
    + \frac{\chi_R'(\|u^\besp{R}_x\|_\alpha)}{\|u^\besp{R}_x\|_\alpha} \scal[V_\alpha]{u^\besp{R}_x, \tu} B(u^\besp{R}_x,u^\besp{R}_x) + {}\\
    + \chi_R(\|u^\besp{R}_x\|_\alpha) [B(\tu,u^\besp{R}_x)+B(u^\besp{R}_x, \tu)]
  = 0,
\end{multline}
with initial condition $\tu(0)=h$, and so
\[
\begin{aligned}
  \frac{d}{dt}\|\tu\|_\alpha^2 + 2\nu\|\tu\|_{\alpha+1}^2
    &\leq    2\bigl|\chi_R'(\|u^\besp{R}_x\|_\alpha) \scal[V_\alpha]{\tu,B(u^\besp{R}_x,u^\besp{R}_x)}\bigr| \|\tu\|_\alpha\\
    &\quad + 2\chi_R(\|u^\besp{R}_x\|_\alpha)|\scal[V_\alpha]{\tu,B(\tu,u^\besp{R}_x)+B(u^\besp{R}_x,\tu)}|.
\end{aligned}
\]
In short, everything boils down to estimating the right-hand side (briefly
denoted below by \term{r}). By Lemma \ref{l:Bnostro} (with $a=b=\alpha$ and $c=-\alpha$)
and Young's inequality,
\[
  \term{r}
    \leq   \frac{c}{R} R^2 \|\tu\|_\alpha \|\tu\|_{\alpha+1}
         + c R \|\tu\|_\alpha \|\tu\|_{\alpha+1}
    \leq   \nu\|\tu\|_{\alpha+1}^2 + \frac{c}{\nu}R^2 \|\tu\|_\alpha^2
\]
and so, by Gronwall's lemma,
\[
  \E\Bigl[\int_0^t\|\tu\|_{\alpha+1}^2\de[s]\Bigr]
    \leq \frac1\nu \|h\|_{\alpha}^2\e^{\frac{c}{\nu}R^2t},
\]
which is enough to bound \eqref{e:BEL}, as, by the choice of $\alpha$,
$1+\alpha\geq\frac32+2\alpha_0$.
\end{proof}
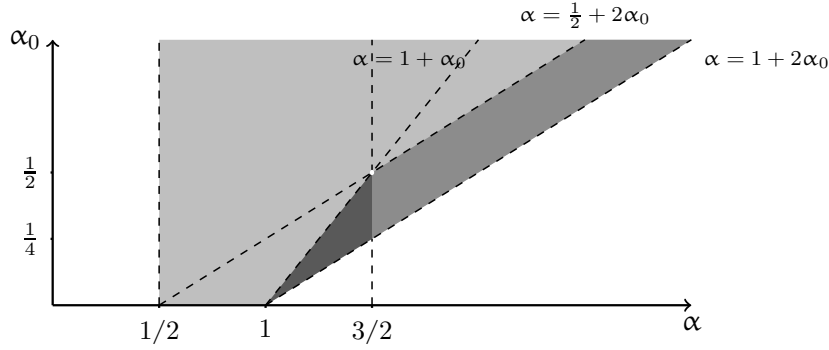
\begin{figure}[h]
  \centering
  \begin{tikzpicture}[x=28mm,y=35.2mm,line width=0.3mm]
    \fill [color=black!25] (0.5,0) -- (1,0) -- (3,1) --(0.5,1);
    \fill [color=black!45] (1.5,0.25) -- (3,1) -- (2.5,1) -- (1.5,0.5);
    \fill [color=black!65] (1,0) -- (1.5,0.25) -- (1.5,0.5);
    \draw [->] (0,0) -- (3,0) node [anchor=north] {\small $\alpha$};
    \draw [->] (0,0) -- (0,1) node [anchor=east] {\small $\alpha_0$};
    \draw (0,0.25) -- (-0.01,0.25) node [anchor=east] {\footnotesize $\tfrac14$};
    \draw (0, 0.5) -- (-0.01, 0.5) node [anchor=east] {\footnotesize $\tfrac12$};
    \draw (0.5, 0) -- (0.5, -0.01) node [anchor=north] {\footnotesize $1/2$};
    \draw (1, 0) -- (1, -0.01) node [anchor=north] {\footnotesize $1$};
    \draw (1.5, 0) -- (1.5, -0.01) node [anchor=north] {\footnotesize $3/2$};
    \draw [dashed,line width=0.2mm] (0.5, 0) -- (0.5, 1);
    \draw [dashed,line width=0.2mm] (1.5, 0) -- (1.5, 0.25);
    \draw [dashed,line width=0.2mm] (1.5, 0.5) -- (1.5, 1);
    \draw [dashed,line width=0.2mm] (1, 0) -- (2, 1) node [anchor=north east] {\tiny $\alpha=1+\alpha_0$};
    \draw [dashed,line width=0.2mm] (0.5, 0) -- (2.5,1) node [anchor=south] {\tiny $\alpha=\frac12+2\alpha_0$};
    \draw [dashed,line width=0.2mm] (1, 0) -- (3, 1) node [anchor=north west] {\tiny $\alpha=1+2\alpha_0$};
    \fill [color=white] (1.5,0.5) circle (0.01);
  \end{tikzpicture}
\caption{The gray areas correspond to existence (Theorem \ref{t:Rexistuniq}),
  the slightly darker gray area corresponds to Proposition~\ref{p:RSFhigh}),
  the darkest area corresponds to Proposition~\ref{p:RSFlow}.}
\end{figure}
\begin{proposition}\label{p:RSFlow}
Assume {\nuno} and {\ndue} of Assumption~\ref{a:mainass}. Given $R\geq1$ and
$\alpha$ such that
\begin{equation}\label{e:AlfaCond3}
  \alpha<\frac32
    \qquad\text{and}\qquad
  1 + \alpha_0 \leq \alpha < 1+2\alpha_0,
\end{equation}
the transition semigroup $(\semi_t^\besp{\alpha,R})_{t\geq0}$
associated to problem \eqref{e:strongR} is $V_\alpha$-strong Feller
for all $t>0$. Moreover, there are numbers $c_1>0$ and $c_2>0$ such that for every
$x_0\in V_\alpha$, for every $\varphi\in B_b(H)$ and for every $h\in V_\alpha$,
\begin{equation}\label{e:RSFlow}
|\semi^\besp{\alpha,R}_t\varphi(x_0+h)-\semi^\besp{\alpha,R}_t\varphi(x_0)|
\leq \frac{c_1}{t\sqrt{\nu}}\|h\|_{\frac12+2\alpha_0}\exp\Bigl(c_2 t \Bigl(\frac{R^4}{\nu^{(2\alpha+1-4\alpha_0)}}\Bigr)^{\frac1{3+4\alpha_0-2\alpha}}\Bigr).
\end{equation}
The strong Feller property as well as formula~\eqref{e:RSFlow} are also true
if $\alpha=\tfrac32$ and $\alpha_0\in(\tfrac14,\tfrac12)$.
\end{proposition}
\begin{proof}
Let $\alpha$ be as in condition \eqref{e:AlfaCond3} and set
$\gamma=2\alpha_0+\frac12$. Fix $x\in V_\alpha$ and $h\in V_\alpha$, and let
$\tu = D_h u_x^\besp{R}$ be the derivative of the flow along $h$, where
$u_x^\besp{R}$ is the solution to problem~\eqref{e:strongR} starting at $x$.
We proceed as in the proof of the previous proposition, so that we only need
to estimate the right-hand side of~\eqref{e:BEL}. Again, $\tu$ solves
\eqref{e:flowderivative}, but we estimate $\tu$ in $V_\gamma$.
Since $\alpha\geq 1+\alpha_0$, we can use Lemma~\ref{l:Bnostro} with $a=b=\alpha$
and $c=-\gamma$, together with interpolation of $V_\alpha$
between $V_\gamma$ and $V_{\gamma+1}$ and Young's inequality to get
\begin{align*}
  \frac{d}{dt}\|\tu\|_\gamma^2 + 2\nu\|\tu\|_{\gamma+1}^2
   &\leq    2\bigl|\chi_R'(\|u^\besp{R}_x\|_\alpha) \scal[V_\gamma]{\tu, B(u^\besp{R}_x,u^\besp{R}_x)}\bigr|\,\|\tu\|_\alpha\\
   &\quad + 2\chi_R(\|u^\besp{R}_x\|_\alpha)|\scal[V_\gamma]{\tu, B(\tu,u^\besp{R}_x)+B(u^\besp{R}_x,\tu)}|\\
   &\leq c R \|\tu\|_\alpha \|\tu\|_{\gamma+1}\\
   &\leq \nu\|\tu\|_{\gamma+1}^2 + c(\nu^{-(1+\alpha-\gamma)} R^2)^{\frac1{1+\gamma-\alpha}}\|\tu\|_\gamma^2,
\end{align*}
and \eqref{e:RSFlow} follows as in the previous theorem.

In the case $\alpha=\tfrac32$ we can choose $\epsilon\in(0,1-2\alpha_0)$ and use
Lemma~\ref{l:Bnostro} with $a=b=\tfrac\epsilon2$ and $c=-\gamma$, with the same
value $\gamma=2\alpha_0+\frac12$.
\end{proof}
\begin{remark}
The conclusions of the previous theorem imply that $(\semi_t^\besp{\alpha,R})_{t\geq0}$
extends to a semigroup on $V_\gamma$ (with a more careful estimate this can be
seen to be true also in the range of values for the parameters $\alpha$, $\alpha_0$
given in Proposition~\ref{p:RSFhigh}). We shall obtain a stronger result in
Section~\ref{s:critical}.
\end{remark}
\subsection{Short time coupling and weak--strong uniqueness}

We show in this section that it is possible to couple for a short time any
solution to the Navier--Stokes equations~\eqref{e:nse} to the unique solution
to~\eqref{e:strongR}, for suitable values of $\alpha$ and $R$. The length of
the short time is a stopping time whose size depends on the initial condition
and the strength of the noise (see Proposition~\ref{p:blowup}).

Given $\alpha\in(\tfrac12,1+2\alpha_0)$, $x\in V_\alpha$ and and an energy
martingale solution (see Definition~\ref{d:ems}) $\Pb_x$, consider the Wiener
process~\eqref{e:wiener} associated to $\Pb_x$ and the process $z$ solution
to~\eqref{e:stokes}. Equation~\eqref{e:strongRv} has a unique solution
$\Pb_x$--{a.~s.}, hence $u_x^\besp{\alpha,R}=z+v_x^\besp{\alpha,R}$ is well
defined and the unique (path-wise and in law) solution to~\eqref{e:strongR}
on the probability space $(\Ons,\field[B],\Pb_x)$ (in particular, it does
not depend in an essential way from $\Pb_x$).

To summarise, we have realised the solutions $(\xi_t)_{t\geq0}$ and
$(u_x^\besp{\alpha,R})_{t\geq0}$ to~\eqref{e:absnse} and~\eqref{e:strongR}
respectively (with the same noise) as stochastic processes on the probability
space $(\Ons,\field[B],\Pb_x)$. Define now
\begin{equation}\label{e:blowuptime}
  \tau_x^\besp{\alpha,R}(\omega)
    =\inf\{\,t\geq0\,:\,\|u_x^\besp{\alpha,R}(t)\|_\alpha\geq R\,\},
\end{equation}
if the above set is non-empty, and $\tau_x^\besp{\alpha,R}=\infty$ otherwise.
\begin{theorem}[Weak-strong uniqueness]\label{t:weakstrong}
Under {\nuno} in Assumption~\ref{a:mainass}, let $\alpha\in(\tfrac12,1 + 2\alpha_0)$
and $R\geq1$. Given $x\in V_\alpha$, let $\Pb_x$ be any energy martingale
solution starting at $x$ and let $(u_x^\besp{\alpha,R})_{t\geq0}$ be
the process solution to~\eqref{e:strongR} defined above on $(\Ons,\Pb_x)$. Then
\[
  (u_x^\besp{\alpha,R}(t) - \xi_t)\uno_{\{\tau_x^\besp{\alpha,R}\geq t\}} = 0,
    \qquad\Pb_x-\text{a.~s.}
\]
for every $t\geq0$. In particular,
\[
  \E^{\Pb^\besp{\alpha,R}_x}[\varphi(\xi_t)\uno_{\{\tau_x^\besp{\alpha,R}\geq t\}}]
    =\E^{\Pb_x}[\varphi(\xi_t)\uno_{\{\tau_x^\besp{\alpha,R}\geq t\}}],
\]
for every $t\geq0$ and every bounded continuous function $\varphi:H\to\R$, where
$\Pb_x^\besp{\alpha,R}$ is the distribution of $u_x^\besp{\alpha,R}$ on $\Ons$.
\end{theorem}
\begin{proof}
If $\Pb[\tau_x^\besp{\alpha,R}\geq t]=0$, there is nothing to prove, so we
assume that such probability is positive. For simplicity we shall write
$u_R = u_x^\besp{\alpha,R}$, $v_R$ the solution to~\eqref{e:strongRv}
corresponding to $u_R$ and $\tau = \tau_x^\besp{\alpha,R}$.

We know that $u_R(s) - \xi_s = v_R(s) - v(s)$, where $v$ is the solution
to~\eqref{e:Veq}, hence it is sufficient to show that $v_R(t) = v(t)$ on
$\{\tau_R\geq t\}$. By continuity (in $H$ for the weak topology for instance),
it is sufficient to show that $v_R(s) = v(s)$ holds for $s<\tau_R$.
If $s<\tau_R$, $\|u_R\|_\alpha\leq R$ and $\chi_R(\|u_R\|_\alpha)=1$, so 
we only need to prove that $v_R$ is the unique weak solution to~\eqref{e:Veq}
for $s<\tau_R$.

Set $\delta = v_R - v$, then $\delta$ satisfies
\[
  \partial_t\delta + \nu\Delta\delta + B(\delta,u_R) + B(\xi,\delta) = 0,
\]
for $s<\tau_R$. Moreover $\delta$ satisfies the following energy inequality
(with the same set of exceptional times corresponding to $v$),
\[
  \frac12\|\delta(s)\|_H^2
    + \nu\int_0^s\|\delta(r)\|_V^2\de[r]
    + \int_0^s\scal[H]{\delta,B(\delta,u_R)}\de[r]
    \leq 0.
\]
Indeed by definition $v$ satisfies an energy inequality (Definition~\ref{d:ems}),
while by Theorem~\ref{t:Rexistuniq} $v_R$ satisfies an energy equality, so we
are left with the proof of an energy balance for $\scal[H]{v_R,v}$. We postpone
this step to the end of the proof and we first show that $\delta(s)=0$ for all
$s<\tau_R$. To this end, we estimate the nonlinear term in the energy balance
for $\delta$. If $\alpha<\tfrac32$, Lemma~\ref{l:Bnostro}, (with $a=\alpha$,
$b=\tfrac32-\alpha$ and $c=0$) and interpolation yield
\[
  |\scal{\delta,b(\delta,u_r)}|
    \leq c\|\delta\|_V\|\delta\|_{\frac32-\alpha}\|u_R\|_\alpha
    \leq cR\|\delta\|_V^{\frac52-\alpha}\|\delta\|_H^{\alpha-\frac12}
    \leq \nu\|\delta\|_V^2 + c(\nu,R)\|\delta\|_H^2,
\]
and so $\delta(s)=0$ for $s<\tau_R$ by Gronwall's lemma. If $\alpha\geq\tfrac32$
one can proceed similarly using an arbitrary value of $a<\tfrac32$.

To conclude the proof, we need to show that
\begin{multline*}
  \scal[H]{v_R(t),v(t)}
       + 2\nu\int_s^t\scal[V]{v_R,v}\de[r] =\\
    = \scal[H]{v_R(s),v(s)}
       - \int_s^t\scal{v_R,B(u,u)}\de[r]
       - \int_s^t\chi_R(\|u_R\|_\alpha)\scal{B(u_R,u_R)}\de[r].
\end{multline*}
We proceed as in Romito~\cite[Theorem~2.2]{Rom10}. As in the proof of the energy
equality for $v_R$ (see Lemmas~\ref{l:Rweak} and~\ref{l:Rmild}), everything boils
down in proving that $\scal[H]{v_R(t),v(t)}$ is differentiable in time with
derivative $\scal{\dot v_R,v}+\scal{v_R,\dot v}$. First we notice that both the
equations for $v$ and $v_R$ are satisfied in $V'$. Moreover we see by the proof
of Lemmas~\ref{l:Rweak} and~\ref{l:Rmild} that $\dot v_R\in L^2_\loc(0,\infty;V')$,
hence $\scal[V',V]{\dot v_R,v}$ is well defined. On the other hand, since by
Corollary~\ref{c:Bnostro} (with $a=1$, $b=0$)
$B(v+z,v+z)\in L^2_\loc(0,\infty;V_{-\beta})$ for all $\beta>\tfrac32$
and either $v_R\in L^2_\loc(0,\infty;V_{\alpha+1})$ (in the range of values of
Lemma~\ref{l:Rweak}) or, by~\eqref{e:mildbound}, $v_R\in L^2_\loc(0,\infty;V_\beta')$
for all $\beta<\alpha+1$ (in the range of values of Lemma~\ref{l:Rmild}),
it turns out that $\scal[V_\beta,V_{-\beta}]{v_R,\dot v}$ is also well defined
and in conclusion $\scal[H]{v_R(t),v(t)}$ is differentiable. The balance above
then follows by the properties of the nonlinearity.
\end{proof}
\section{Critical regularity for the strong Feller property}\label{s:critical}

In the previous section we have proved that the transition semigroup associated
to any Markov solution has a regularising effect in strong topologies. Namely,
the semigroup computed on bounded measurable functions gives back almost
Lipschitz functions (see formula~\eqref{e:lip}). In this section we show
that the space where the regularity of the semigroup holds can be relaxed,
at the price of having continuity only. We remark that it may be possible
to achieve strong Feller regularity including the value $\alpha=\tfrac12$,
but this would require some more refined analytical method, which would make
the paper much lengthier.
\begin{theorem}\label{t:main}
Under Assumption~\ref{a:mainass}, let $(\semi_t)_{t\geq0}$ be the transition
semigroup associated to a Markov solution $(\Pb_x)_{x\in H}$.
Then $(\semi_t)_{t\geq0}$ is $V_\alpha$-strong Feller for every
$\alpha>\tfrac12$.
\end{theorem}
The theorem follows from Theorem~\ref{t:sf} and Proposition~\ref{p:critonestep}
below, which contains the core idea. We first prove the following convergence
lemma on the approximated problem examined in Appendix~\ref{ss:approx}.
\begin{lemma}\label{l:parabRreg}
Assume {\nuno} (from Assumption~\ref{a:mainass}) and let
$\alpha\in(\tfrac12,1+2\alpha_0)$ and $\beta\in(\alpha,1+2\alpha_0)$ such that
$\beta<\alpha+(\tfrac12\wedge(\alpha-\tfrac12))$. If $x_n\to x$ in $V_\alpha$
and $R\geq1$, then $u_{x_n}^\besp{\alpha,R}(t)\to u_x^\besp{\alpha,R}(t)$
almost surely in $V_\beta$ for all $t>0$, where $u^\besp{\alpha,R}_y$ is
the solution to~\eqref{e:strongR} with initial condition $y$.
\end{lemma}
\begin{proof}
Denote for simplicity $u_n=u^\besp{\alpha,R}_{x_n}$ and $u=u^\besp{\alpha,R}_x$.
Let $z$ be the solution to the Stokes problem~\eqref{e:stokes} and set $v_n=u_n-z$,
$v=u-z$ and $w_n=u_n-u$, which solves the following equation,
\begin{multline*}
\dot w_n
  + \nu Aw_n
  + \chi_R(\|u_n\|_\alpha) B(u_n, w_n)
  + \chi_R(\|u\|_\alpha) B(w_n, u) + {}\\
  + \bigl(\chi_R(\|u_n\|_\alpha) - \chi_R(\|u\|_\alpha)\bigr) B(u_n,u)
  = 0.
\end{multline*}
Assume first that $\beta<\tfrac32$, then
\begin{multline*}
  \|w_n(t)\|_\beta
    \leq  \|\e^{-\nu At}w_n(0)\|_\beta
        + \int_0^t \Bigl(\chi_R(\|u_n\|_\alpha) \|\e^{-\nu A(t-s)} B(u_n, w_n)\|_\beta + {}\\
        + |\chi_R(\|u\|_\alpha) - \chi_R(\|u_n\|_\alpha)| \|\e^{-\nu A(t-s)} B(u_n,u)\|_\beta + {}\\
        + \chi_R(\|u\|_\alpha) \|\e^{-\nu A(t-s)} B(w_n, u)\|_\beta\Bigr)\de[s].
\end{multline*}
We use Corollary~\ref{c:Bnostro} (with $a=\alpha$, $b=\beta$ for the first
two terms in the integral and $a=b=\alpha$ for the third term) and
properties~\eqref{e:semiprop} and \eqref{e:chiprop} to get
\[
  \|w_n(t)\|_\beta
    \leq  ct^{-\frac12(\beta-\alpha)}\|x_n-x\|_\alpha
        + c_R \bigl(1+t^{\frac12(\beta-\alpha)}\bigr)\int_0^t (t-s)^{-\frac{2\beta+5-4\alpha}{4}}\|w_n(s)\|_\beta\de[s].
\]
Notice that the assumptions on $\beta$ ensure that
$\tfrac14(2\beta+5-4\alpha)<1$. Fix $T>0$ and let $a_\beta$ be the weight
function in Lemma~\ref{l:mildbound} (with $x=\tfrac12(\beta-\alpha)$ and
$y=\tfrac14(2\beta+5-4\alpha)$) so that
\[
  c_R \bigl(1 + T^{\frac12(\beta-\alpha)}\bigr) a_\beta(t)
        \int_0^t (t-s)^{-\frac14(2\beta+5-4\alpha)}a_\beta(s)^{-1}\de[s]
    \leq\frac12.
\]
With this choice, $\sup_{s\leq T}a_\beta(s)\|w_n(s)\|_\beta\leq c_{R,T}\|x_n-x\|_\alpha$
and so $\|w_n(t)\|_\beta\to0$ for $t>0$.

Consider now the case $\beta>\tfrac32$ (in particular this implies that
$\alpha$ is in the range of Lemma~\ref{l:Rweak}). The energy estimate,
Lemma~\ref{l:Bnostro} (with $a=b=\beta$, $c=-\beta$), formula~\eqref{e:chiprop}
and Young's inequality yield
\[
  \frac{d}{dt}\|w_n\|_\beta^2 
    \leq 
         c_R(1+\|z\|_\beta)^4(1 + \|v_n\|_{\alpha+1}^2
       + \|v\|_{\alpha+1}^2)^{2(\beta-\alpha)}\|w_n\|_\beta^2,
\]
since $\|u\|_\beta\leq \|z\|_\beta+(\|u\|_\alpha+\|z\|_\alpha)^{1+\alpha-\beta}\|v\|_{\alpha+1}^{\beta-\alpha}$
by interpolation of $V_\beta$ between $V_\alpha$ and $V_{\alpha+1}$ (similarly
for $u_n$). By assumption $2(\beta-\alpha)<1$, hence Gronwall's lemma
implies that for all $s\leq t$,
\[
  \|w_n(t)\|_\beta^2
    \leq \|w_n(s)\|_\beta^2\exp\Bigl(c_R\int_s^t (1+\|z\|_\beta)^4(1 + \|v_n\|_{\alpha+1}^2 + \|v\|_{\alpha+1}^2)^{2(\beta-\alpha)}\de[r]\Bigr).
\]
By integrating for $s\in[0,\tfrac{t}{2}]$, we get
\[
  \|w_n(t)\|_\beta^2
    \leq \frac2{t}\Bigl(\int_0^t \|w_n(s)\|_\beta^2\Bigr)
           \exp\Bigl(c_R\int_0^t (1+\|z\|_\beta)^4(1 + \|v_n\|_{\alpha+1}^2 + \|v\|_{\alpha+1}^2)^{2(\beta-\alpha)}\Bigr).
\]
The exponential term is uniformly bounded in $n$ (using inequality%
~\eqref{e:gt32}), so we only need to show that the first integral on
the right hand side converges to zero. If $\beta\leq\alpha+\tfrac14$
the result follows by applying inequality \eqref{e:uniqgt32} to
$w_n = v_n-v$. On the other hand, if $\beta>\alpha+\frac14$, interpolation
(between $V_{\alpha+\frac14}$ and $V_{\alpha+1}$) ensures convergence
since, as above, $\int\|w_n\|^2_{\alpha+1/4}\to0$ and $w_n$ is
bounded uniformly in $n$ in $L^2(0,t;V_{\alpha+1})$ (this can be proved
using~\eqref{e:gt32} on both $v_n$ and $v$).

Finally, if $\beta=\tfrac32$, one can consider a slightly larger value
$\beta'>\beta$ which satisfies the same assumptions of $\beta$ and apply the
computations above.
\end{proof}
\begin{proposition}\label{p:critonestep}
Assume {\nuno} of Assumption~\ref{a:mainass} and let $(\semi_t)_{t\geq0}$ be
the transition semigroup associated to a Markov solution $(\Pb_x)_{x\in H}$
to~\eqref{e:nse}. If $\alpha\in(\tfrac12,1+2\alpha_0)$ and there is a number
$\beta\in(\alpha,1+2\alpha_0)$ such that $(\semi_t)_{t\geq0}$ is $V_\beta$-strong
Feller, then $(\semi_t)_{t\geq0}$ is $V_\alpha$-strong Feller.
\end{proposition}
\begin{proof}
It is sufficient to show the theorem under the condition
$\beta<\alpha + (\tfrac12\wedge(\alpha-\tfrac12))$. The general case
follows by iterating the argument.

Let $x_n\to x$ in $V_\alpha$. Choose $R\geq 1+4\sup_n\|x_n\|_\alpha$
and $\epsilon_0\leq c'R^{-\gamma}$, where $c'$, $\gamma$, $\eta$ are
the values given in Proposition~\ref{p:blowup}. With such values, we
know that, by Proposition~\ref{p:blowup},
\[
  \{\sup_{t\in[0,\epsilon_0]}\|z(t)\|_\eta\leq\frac{R}3\}
    \subset A_\epsilon
    =       \{\tau_x^\besp{\alpha,R}\geq\epsilon\}\cap\bigcap_{n\in\N}\{\tau_{x_n}^\besp{\alpha,R}\geq\epsilon\}
\]
for every $\epsilon\leq\epsilon_0$, where $\tau^\besp{\alpha,R}$ is
defined in~\eqref{e:blowuptime}. Notice that for any $\varphi\in B_b(H)$
and $\epsilon\leq\epsilon_0$ (so that it does not belong to any of the
exceptional sets of $\Pb_{x_n}$, $\Pb_x$), by the Markov property and
Theorem~\ref{t:weakstrong},
\[
\begin{aligned}
  \semi_t\varphi(y)
  & =  \E^{\Pb_y}[\semi_{t-\epsilon}\varphi(\xi_\epsilon)\uno_{A_\epsilon}]
     + \E^{\Pb_y}[\semi_{t-\epsilon}\varphi(\xi_\epsilon)\uno_{A_\epsilon^c}]\\
  & =  \semi^\besp{\alpha,R}_\epsilon(\semi_{t-\epsilon}\varphi)(y)
     + \E^{\Pb_y}[\semi_{t-\epsilon}\varphi(\xi_\epsilon)\uno_{A_\epsilon^c}]
     - \E^{\Pb_y^\besp{\alpha,R}}[\semi_{t-\epsilon}\varphi(\xi_\epsilon)\uno_{A_\epsilon^c}],
\end{aligned}
\]
with $y=x_n$ or $y=x$, where $(\semi^\besp{\alpha,R}_t)_{t\geq0}$ is
the transition semigroup associated to problem~\eqref{e:strongR}. Since
by Lemma~\ref{l:Ztails} the term
\begin{multline*}
  |o_{\epsilon,R}(y)|
     = \bigl|\E^{\Pb_y}[\semi_{t-\epsilon}\varphi(\xi_\epsilon)\uno_{A_\epsilon^c}]
       - \E^{\Pb_y^\besp{\alpha,R}}[\semi_{t-\epsilon}\varphi(\xi_\epsilon)\uno_{A_\epsilon^c}]\bigr|
    \leq\\
    \leq 2\|\varphi\|_\infty\Pb_y^\besp{\alpha,R}[A_\epsilon^c]
    \leq 2\|\varphi\|_\infty\Pb_y^\besp{\alpha,R}[\sup_{t\leq\epsilon_0}\|z(t)\|_\eta\geq\frac{R}{3}]
    \leq c\|\varphi\|_\infty\e^{-a_0\frac{R^2}{\epsilon_0}}
\end{multline*}
converges to $0$ as $\epsilon_0\to0$ uniformly in $n$, we have that
\[
\semi_t\varphi(x_n) - \semi_t\varphi(x)
 =   \semi_\epsilon^\besp{\alpha,R}(\semi_{t-\epsilon}\varphi)(x_n) - \semi_\epsilon^\besp{\alpha,R}(\semi_{t-\epsilon}\varphi)(x)
   + o_{\epsilon,R}(x_n) - o_{\epsilon,R}(x).
\]
By assumptions, $\semi_{t-\epsilon}\varphi\in C_b(V_\beta)$, and by
Lemma~\ref{l:parabRreg} $u_{x_n}^\besp{\alpha,R}(\epsilon)\to u_{x_n}^\besp{\alpha,R}(\epsilon)$
almost surely, where $u_y^\besp{\alpha,R}$ is the solution to~\eqref{e:strongR}
with initial condition $y$. By Lebesgue theorem
$\semi_\epsilon^\besp{\alpha,R}(\semi_{t-\epsilon}\varphi)(x_n)\to \semi_\epsilon^\besp{\alpha,R}(\semi_{t-\epsilon}\varphi)(x)$
as $n\to\infty$, and, in the limit as $\epsilon_0\to0$, we have that
$\semi_t\varphi(x_n)\to \semi_t\varphi(x)$.
\end{proof}
\subsection{A few consequences}\label{ss:consequence}

As a preliminary result we show that under {\ndue} (see Assumption~\ref{a:mainass})
each Markov solutions has Markov kernels supported on the whole state space.
We follow the lines of~\cite{Fla97}. For stronger results on the same lines
we refer to \cite{Rom04,Rom05,AgrSar05,Shi06,AgrKukSarShi07,RomXu09}.
\begin{lemma}\label{l:control}
Under {\ndue} consider a Markov solution $(\Pb_x)_{x\in H}$. Then for every
$\tfrac12<\alpha<1+2\alpha_0$, every $x\in V_\alpha$, every $t>0$ and every open
set $U\subset V_\alpha$, $P(t,x,U)>0$, where $P(\cdot,\cdot,\cdot)$ is the Markov
kernel associated to the given Markov solution.
\end{lemma}
\begin{proof}
Without loss of generality, we can assume $\alpha>2\alpha_0$. We proceed as in
\cite[Proposition~6.1]{FlaRom08}: we need to show that
$\Pb_x[\|\xi_t-y\|_\alpha<\epsilon]>0$ for all $t>0$, $x,y\in V_\alpha$. This
probability is bounded from below by 
$\Pb_x^\besp{\alpha,R}[\|\xi_t-y\|_\alpha<\epsilon,\tau^\besp{\alpha,R}>t]$,
hence it is sufficient to show that this last quantity is positive. This follows
by solving a control problem as in Lemmas~C.2,~C.3 of \cite{FlaRom08}. 
\end{proof}
\begin{corollary}
Under Assumption~\ref{a:mainass}, every Markov solution $(\Pb_x)_{x\in H}$
to~\eqref{e:nse} admits a unique invariant measure, which is strongly mixing.
Moreover, the convergence to the invariant measure is exponentially fast.

Finally, if $(\Pb_x^1)_{x\in H}$ and $(\Pb_x^2)_{x\in H}$ are different
Markov solutions, then the corresponding Markov kernels $P_1(t,x,\cdot)$
and $P_2(t,x,\cdot)$ are equivalent measures for all $x\in V_\alpha$ and
$\alpha>\tfrac12$. Equivalence holds also for the corresponding invariant
measures.
\end{corollary}
\begin{proof}
Given the above lemma, unique ergodicity is a consequence of strong Feller
regularity and Doob's theorem (see~\cite{DapZab92}). This extends
\cite[Corollary~3.2]{Rom08}. Exponential convergence is an extension of
\cite[Theorem~3.3]{Rom08} and follows with similar methods.
Finally, equivalence of laws follows as in~\cite[Theorem~4.1]{FlaRom07}.
\end{proof}
We finally give a generalisation of Theorem~6.7 of~\cite{FlaRom08}.
\begin{proposition}
Under Assumption~\ref{a:mainass}, let $(\Pb_x)_{x\in H}$ be a Markov solution
to~\ref{e:nse}. Then for any $\alpha>\tfrac12$, $(\Pb_x)_{x\in V_\alpha}$ is
a Markov family.
\end{proposition}
\begin{proof}
We prove preliminarily the following claim:  for every $\alpha>\tfrac12$, $t_0>0$
and $x\in V_\alpha$, $\xi$ is continuous with values in $V_\alpha$ in a
neighbourhood of $t_0$, $\Pb_x$--{a.~s.}. Indeed, once this claim is proved,
the proposition follows as in~\cite[Theorem~6.7]{FlaRom08}, since the only
necessary ingredient is that the transition semigroup is strong Feller.

Let $\mu$ be the unique invariant measure of $(\Pb_x)_{x\in H}$ and let $\Pb^\star$
be the corresponding stationary solution (that is, the solution starting at $\mu$).
We notice that, by \cite[Corollary~3.2]{Rom08} (which depends only on Theorem A.2
in the same paper and whose assumption is {\nuno}), for every $\beta<1+2\alpha_0$
there is $\eta = \eta(\beta)>0$ such that $\E^{\mu}\|x\|_\beta^\eta<\infty$.

Fix $\alpha>\tfrac12$, $t_0>0$ and $x\in V_\alpha$. For every $0<a<b$, set
$A(a,b) = C((a,b);V_\alpha)$, we wish to show that
$\Pb_x[\xi\in\bigcup_\epsilon A(t_0-\epsilon,t_0+\epsilon)]=1$. By the Markov property,
\begin{multline*}
\Pb^\star[A(t_0-\epsilon,t_0+\epsilon)]
  \geq \Pb^\star[\|\xi_{t_0-2\epsilon}\|_\alpha\leq\tfrac{R}3]\inf_{\|y\|_\alpha\leq\frac{R}3}\Pb_y[\xi\in A(\epsilon,3\epsilon)]
  \geq\\
  \geq \Bigl(1-\frac{c}{R^\eta}\E^\mu[\|x\|_\alpha^\eta]\Bigr)\inf_{\|y\|_\alpha\leq\frac{R}3}\Pb_y[\xi\in A(\epsilon,3\epsilon)]
\end{multline*}
Using Theorem~\ref{t:weakstrong} and taking $\epsilon\leq cR^{-\gamma}$ (where $c$,
$\gamma$ are from Proposition~\ref{p:blowup}), we have
\begin{multline*}
  \Pb_y[\xi\in A(\epsilon,3\epsilon)]
    = \Pb_y^\besp{\alpha,R}[\xi\in A(\epsilon,3\epsilon)]
      + {}\\
      + \bigl(\Pb_y[\xi\in A(\epsilon,3\epsilon),\tau^\besp{\alpha,R}<3\epsilon]
            - \Pb_y^\besp{\alpha,R}[\xi\in A(\epsilon,3\epsilon),\tau^\besp{\alpha,R}<3\epsilon]).
\end{multline*}
Clearly, $\Pb_y^\besp{\alpha,R}[\xi\in A(\epsilon,3\epsilon)]=1$, while the last
term on the right hand side converges to $0$ for $\epsilon\downarrow0$ and
$R\uparrow\infty$. In conclusion
$\inf_{\|y\|_\alpha\leq\frac{R}3}\Pb_y[\xi\in A(\epsilon,3\epsilon)]\to0$
and $\Pb^\star[\xi\in\bigcup_\epsilon A(t_0-\epsilon,t_0+\epsilon)]=1$. In
particular $\Pb_x[\xi\in\bigcup_\epsilon A(t_0-\epsilon,t_0+\epsilon)]=1$
for $\mu$--{a.~e.} $x$, hence for all $x$ by the strong Feller property
and Lemma~\ref{l:control}.
\end{proof}
\section{Technical tools}\label{s:technical}

\subsection{Short time coupling with a smooth problem}\label{ss:approx}

We follow the approach of \cite{FlaMas95} (see also~\cite{FlaRom08,Rom08}) to
construct a regular process which coincides with any solution to~\eqref{e:nse}
for a short time, using a cut-off of the nonlinearity. In this way with large
probability the two solutions have the same trajectories on a small time
interval.
\subsubsection{Existence for the regular problem}

Let $\chi:[0,\infty]\to[0,1]$ be a non-increasing $C^\infty$ function such
that $\chi\equiv 1$ on $[0,1]$ and $\chi_R\equiv 0$ on $[2,\infty)$ (see
Figure~\ref{f:chi}). Given $R\geq1$, set $\chi_R(x) = \chi(\tfrac{x}{R})$.
\begin{figure}[h]
  \centering
  \begin{tikzpicture}[x=1.5mm,y=1mm,line width=0.3mm]
    \draw [->,line width=0.1mm] (0,0) -- (32,0);
    \draw [->,line width=0.1mm] (0,0) -- (0,15);
    \draw (0,12)--(12,12) to [out=0,in=180] (24,0) -- (30,0);
    \draw [dashed,line width=0.1mm] (12,0) -- (12,12); 
    \draw node at (-1,12) {$1$};
    \draw node at (12,-3) {$1$};
    \draw node at (24,-3) {$2$};
  \end{tikzpicture}
\caption{The cut-off function $\chi$}\label{f:chi}
\end{figure}
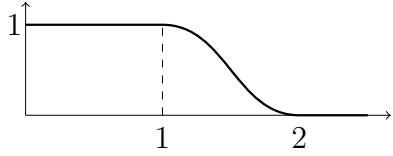
Consider the following problem,
\begin{equation}\label{e:strongR}
  \begin{cases}
    du + \nu Au\de + \chi_R(\|u\|_\alpha)B(u,u)\de = \cov^{\frac12}\de[W],\\
    u(0)=x.
  \end{cases}
\end{equation}
In the following we analyse for which values of ($\alpha,\alpha_0)$
the above problem is uniquely solvable.
\begin{theorem}\label{t:Rexistuniq}
Assume \nuno (Assumption~\ref{a:mainass}). Given $R\geq1$ and $\tfrac12<\alpha<1+2\alpha_0$,
for every $x\in V_\alpha$ problem \eqref{e:strongR} has a path-wise unique
martingale solution $\Pb^\besp{\alpha,R}_x$ on $\Ons$, with
\begin{equation}\label{e:Rcontinua}
\Pb^\besp{\alpha,R}_x[C([0,\infty);V_\alpha)]=1.
\end{equation}
Moreover, $(\Pb^\besp{\alpha,R}_x)_{x\in V_\alpha}$ is a Markov family and
its transition semigroup is Feller on $V_\alpha$.
Finally, for every $0\leq s<t$,
\begin{equation}\label{e:Rei}
  \frac12\|v_t\|_H^2
    + \nu\int_s^t\|v_r\|_V^2\de[r]
    - \int_s^t\chi_R(\|v_r+z_r\|_\alpha)\scal{z_r, B(v_r+z_r, v_r)}\de[r]
    = \frac12\|v_s\|_H^2,
\end{equation}
$\Pb^\besp{\alpha,R}_x$--{a.~s.}, where $z$ is the solution to~\eqref{e:stokes}
and $v$ solves~\eqref{e:strongRv} below.
\end{theorem}
\begin{remark}
The two bounds on $\alpha$ required in the assumptions of the above theorem
have a different justification. The requirement $\alpha<1+2\alpha_0$ is due
to the fact that the linearisation at $0$ (that is, problem~\eqref{e:stokes})
has that maximal regularity (see for instance~\cite{DapZab92}). On the other
hand, $\alpha>\tfrac12$ because $H^{1/2}$ is the largest space in the
Sobolev--Hilbert hierarchy of spaces (see \cite{FujKat64}).
\end{remark}
We give a short sketch of the proof of the above theorem, which can be
made rigorous by using suitable approximations (such as Galerkin approximations)
as in the proof of existence for the Navier-Stokes equations themselves (see for
instance \cite{FlaGat95}).

Let $z$ denote the solution to the Stokes problem~\eqref{e:stokes}
starting at $0$. By the assumption on $\cov$, trajectories of the noise
belong to $C^\gamma([0,\infty);V_{\alpha'})$ for all $\gamma\in[0,\frac12)$
and all $\alpha'<2\alpha_0$. Hence, with probability one,
$z\in C([0,\infty);V_{1+2\alpha_0-\epsilon}))$, for all $\epsilon>0$.
In particular, $z\in C([0,\infty);V_\alpha)$ with probability one.

Fix $\alpha$, $R\geq1$ and $x\in V_\alpha$ and write $u = v + z$,
where $v$ is the solution to
\begin{equation}\label{e:strongRv}
  \partial_t v + \nu Av + \chi_R(\|v+z\|_\alpha)B(v+z,v+z)=0.
\end{equation}
with initial condition $v(0)=x$.
\begin{lemma}\label{l:Rweak}
Assume {\nuno} from Assumption~\ref{a:mainass} and
$\tfrac12<\alpha<\min\{\tfrac12+4\alpha_0,1+2\alpha_0)$.
Then for every $x\in V_\alpha$ there is a solution
$v\in C([0,\infty);V_\alpha)\cap L^2_\loc([0,\infty);V_{\alpha+1})$
to problem~\eqref{e:strongRv}. Moreover, $v$ satisfies the
balance~\eqref{e:Rei}.
\end{lemma}
\begin{proof}
For brevity, we only give details of the crucial estimates needed to prove
that \eqref{e:strongRv} can be solved pathwise and has a global weak solution
in $C([0,\infty);V_\alpha)$ and $L^2_\loc([0,\infty);V_{\alpha+1})$.
The energy estimate in $V_\alpha$ yields
\[
  \frac{d}{dt}\|v\|_\alpha^2 + 2\nu\|v\|_{\alpha+1}^2
  \leq 2\chi_R(\|u\|_\alpha)\scal[V_\alpha]{v, B(u, u)}.
\]
If $\alpha>\tfrac32$, using Lemma~\ref{l:Bnostro} (with $a=b=\alpha$
and $c=-\alpha$) and Young's inequality (with exponent $2$),
\begin{equation}\label{e:gt32}
  2\chi_R(\|u\|_\alpha)\scal[V_\alpha]{v, B(u, u)}
  \leq c\chi_R(\|u\|_\alpha)\|v\|_{\alpha+1}\|u\|_\alpha^2
  \leq  \nu\|v\|_{\alpha+1}^2 + cR^4.
\end{equation}
which implies an a-priori estimate in $L^\infty_\loc([0,\infty);V_\alpha)$
and $L^2_\loc([0,\infty);V_{\alpha+1})$.

If $\alpha=\tfrac32$, choose $\epsilon<1$ such that $\alpha+\epsilon<1+2\alpha_0$.
Lemma~\ref{l:Bnostro} ($a=\alpha$, $b=\alpha+\epsilon$, $c=-\alpha$),
interpolation of $V_{\alpha+\epsilon}$ between $V_\alpha$ and $V_{\alpha+1}$,
and Young's inequality (with exponents $2$ and $\tfrac{2}{1+\epsilon}$) yield
\begin{equation}\label{e:equal32}
  \begin{split}
    2\chi_R(\|u\|_\alpha^2)\scal[V_\alpha]{v, B(u, u)}
    &\leq c\chi_R(\|u\|_\alpha)\|v\|_{1+\alpha}\|u\|_\alpha\|u\|_{\alpha+\epsilon}\\
    &\leq cR\|v\|_{1+\alpha}\bigl[\|z\|_{\alpha+\epsilon} + (R + \|z\|_{\alpha})^{1-\epsilon}\|v\|_{\alpha+1}^\epsilon\bigr]\\
    &\leq \nu\|v\|_{\alpha+1}^2 + cR^2\|z\|_{\alpha+\epsilon}^2 + cR^{\frac{2}{1-\epsilon}}(R + \|z\|_\alpha)^2,
  \end{split}
\end{equation}
and again an a-priori estimate for $v$ in $L^\infty_\loc([0,\infty);V_\alpha)$
and $L^2_\loc([0,\infty);V_{\alpha+1})$.

Finally, if $\alpha<\tfrac32$, we use Lemma~\ref{l:Bnostro} ($a=b=\tfrac14(2\alpha+3)$,
$c=-\alpha$), interpolation of $V_{\frac14(2\alpha+3)}$ and Young's inequality,
\begin{equation}\label{e:lt32}
  \begin{split}
    2\chi_R(\|u\|_\alpha)\scal[V_\alpha]{v, B(u, u)}
    &\leq c\chi_R(\|u\|_\alpha)\|v\|_{\alpha+1}\|u\|_{\frac14(2\alpha+3)}^2\\
    &\leq \nu \|v\|_{\alpha+1}^2 + c\|z\|_{\frac14(2\alpha+3)}^4 + c(R + \|z\|_\alpha)^{\frac{2(2\alpha+1)}{2\alpha-1}}.
  \end{split}
\end{equation}
Here we need $\tfrac14(2\alpha+3)<1+2\alpha_0$ (hence $\alpha<\tfrac12+4\alpha_0$),
to have $\|z\|_{\frac14(2\alpha+3)}$ finite.

We also need an a-priori estimate for $\partial_t v$ in $L^2(0,T;V_{\alpha-1})$,
for all $T>0$. This will imply continuity in time of $v$ on $V_\alpha$
(see for instance \cite{Tem77}). Together with continuity of $z$, it
implies \eqref{e:Rcontinua}. To do this, multiply the equations by
$A^{\alpha-1}\dot v$ to get
\[
  2\|\dot v\|_{\alpha-1}^2 + \nu\frac{d}{dt}\|v\|_\alpha^2
    = - 2\chi_R(\|u\|_\alpha)\scal{A^{\alpha-1}\dot v, B(u,u)}.
\]
The right hand side can be estimated in the three cases through Lemma~\ref{l:Bnostro}
as in \eqref{e:gt32}, \eqref{e:equal32} and \eqref{e:lt32} respectively (using
the same values of $a$, $b$, $c$).

Finally, since $\partial_t v\in L^2(0,T;V_{\alpha-1})$, it follows that
equation~\eqref{e:strongRv} is satisfied in $V'$ and $t\mapsto\|v(t)\|_H^2$ is
differentiable with derivative $2\scal[V',V]{\partial_t v,v}$. Equality~\eqref{e:Rei}
follows easily from these two facts and the properties of the nonlinearity.
\end{proof}
\begin{lemma}\label{l:Rmild}
Assume {\nuno} from Assumption~\ref{a:mainass} and let
$\alpha\in(\tfrac12+4\alpha_0,1+2\alpha_0)$. Then
for every $x\in V_\alpha$ there is a solution $v\in C([0,\infty);V_\alpha)$
to problem~\eqref{e:strongRv}. Moreover, $v$ satisfies the balance~\eqref{e:Rei}
and for every $\beta\in(\alpha,1+2\alpha_0)$ and every $T>0$ there is
$c = c(\alpha,\beta,R,T)>0$ such that
\begin{equation}\label{e:mildbound}
  \sup_{t\leq T} (t\wedge1)^{\frac12(\beta-\alpha)}\|v(t)\|_\beta
  \leq c(\|x\|_\alpha + \sup_{t\leq T}\|z(t)\|_\beta).
\end{equation}
\end{lemma}
\begin{proof}
The standard bounds in $L^\infty(0,T;H)$ and $L^2(0,T;V)$ ensure compactness
of approximations (as in standard proofs for Navier--Stokes~\cite{Tem77}).
Convergence in $V_\alpha$ is needed in order to show that any limit point
is a solution. This follows from Ascoli-Arzel\`a theorem. Indeed,
Corollary~\ref{c:Bnostro} (with $a=b=\alpha$) implies that (we omit
the subscript $n$ for simplicity),
\begin{equation}\label{e:mildalpha}
  \begin{split}
    \|v(t)\|_\alpha
    &\leq \|\e^{-\nu At}x\|_\alpha + \int_0^t \chi_R(\|u\|_\alpha)|\e^{-\nu A(t-s)}B(u,u)|_\alpha\de[s]\\
    &\leq \|x\|_\alpha + c\int_0^t (t-s)^{-\frac14(5-2\alpha)}\chi_R(\|u\|_\alpha)\|u\|_\alpha^2\de[s]\\
    &\leq \|x\|_\alpha + c R^2t^{\frac14(2\alpha-1)},
  \end{split}
\end{equation}
where we have used that
\begin{equation}\label{e:semiprop}
  \|A^\gamma\e^{-\nu At}\|_{\mathcal{L}(H)}\leq ct^{-\gamma}.
\end{equation}
Similarly, if $\beta>\alpha$, Corollary~\ref{c:Bnostro} ($a=\alpha$,
$b=\beta$) yields
\[
\begin{split}
  \|v(t)\|_\beta
  &\leq \|\e^{-\nu At}x\|_\beta + \int_0^t \chi_R(\|u\|_\alpha)|\e^{-\nu A(t-s)}B(u,u)|_\beta\de[s]\\
  &\leq c t^{-\frac12(\beta-\alpha)}\|x\|_\alpha + cR\int_0^t (t-s)^{-\frac14(5-2\alpha)}(\|v(s)\|_\beta+\|z(s)\|_\beta)\de[s]\\
  &\leq c t^{-\frac12(\beta-\alpha)}\|x\|_\alpha + c Rt^{\frac{2\alpha-1}{4}}\sup_{s\leq T}\|z(t)\|_\beta + cR\int_0^t (t-s)^{-\frac{5-2\alpha}{4}}\|v(s)\|_\beta\de[s].
\end{split}
\]
Choose $a_\beta(t)$ as in Lemma~\ref{l:mildbound} so that
\[
  c R a_\beta(t)\int_0^t (t-s)^{-\frac14(5-2\alpha)} a_\beta(s)^{-1}\de[s]\leq \frac12,
\]
hence
\begin{equation}\label{e:mildhigher}
  \sup_{t\leq T} a_\beta(t)\|v(t)\|_\beta
  \leq c\|x\|_\alpha + c\sup_{t\leq T}\|z(t)\|_\beta.
\end{equation}
Equicontinuity in time can be obtained by an estimate similar to~\eqref{e:mildalpha},
hence  there is a subsequence of $(v_n)_{n\in\N}$ converging uniformly in
$V_\alpha$ on any interval $[\epsilon,T]$. In particular, this implies that the
limit point is a solution to \eqref{e:strongRv} and it is continuous in
$V_\alpha$ on $(0,T]$. Continuity in $0$ can be obtained with an estimate
similar to \eqref{e:mildalpha}.
Finally, the bounds~\eqref{e:mildbound} can be obtained as in~\eqref{e:mildhigher}
and in turns they imply uniqueness, via Lemma~\ref{l:Runiq} below.

Finally, we prove the energy balance~\eqref{e:Rei}. The estimate~\eqref{e:mildbound} implies that
$Av\in L^2_\loc(0,\infty;V')$, while by Lemma~\ref{l:Bnostro} (with $a=\alpha$,
$b=1$ and $c=0$) we know that $\|\chi_R(\|u\|_\alpha) B(u,u)\|_{V'}\leq cR\|u\|_V$,
hence $\chi_R(\|u\|_\alpha) B(u,u)$ is in $L^2_\loc(0,\infty;V')$ and in conclusion
$\partial_t v\in L^2_\loc(0,\infty;V')$ and equality~\eqref{e:strongRv} holds
in $V'$. Equality~\eqref{e:strongRv} again follows easily from these two facts
and the properties of the nonlinearity.
\end{proof}
\begin{figure}[h]
  \centering
  \begin{tikzpicture}[x=28mm,y=35.2mm,line width=0.3mm]
    \fill [color=black!25] (0.5,0) -- (0.5,1) -- (3,1) --(1.5,0.25) -- (0.5,0);
    \fill [color=black!45] (0.5, 0) -- (1.5, 0.25) -- (1, 0);
    \draw [->] (0,0) -- (3,0) node [anchor=north] {\small $\alpha$};
    \draw [->] (0,0) -- (0,1) node [anchor=east] {\small $\alpha_0$};
    \draw (0,0.25) -- (-0.01,0.25) node [anchor=east] {\footnotesize $\tfrac14$};
    \draw (0, 0.5) -- (-0.01, 0.5) node [anchor=east] {\footnotesize $\tfrac12$};
    \draw (0.5, 0) -- (0.5, -0.01) node [anchor=north] {\footnotesize $1/2$};
    \draw (1, 0) -- (1, -0.01) node [anchor=north] {\footnotesize $1$};
    \draw (1.5, 0) -- (1.5, -0.01) node [anchor=north] {\footnotesize $3/2$};
    \draw [dashed,line width=0.2mm] (0.5, 0) -- (0.5, 1);
    \draw [dashed,line width=0.2mm] (1.5, 0) -- (1.5, 0.25);
    \draw [dashed,line width=0.2mm] (0.5, 0) -- (3,0.625) node [anchor=south] {\tiny $\alpha=\frac12+4\alpha_0$};
    \draw [dashed,line width=0.2mm] (1, 0) -- (3, 1) node [anchor=south] {\tiny $\alpha=1+2\alpha_0$};
  \end{tikzpicture}
\caption{The coloured area corresponds to all pairs of parameters $\alpha$,
  $\alpha_0$ where existence and uniqueness for~\eqref{e:strongR} holds.
  The light gray area is Lemma \ref{l:Rweak}, the dark gray area is
  Lemma~\ref{l:Rmild}.}
\end{figure}
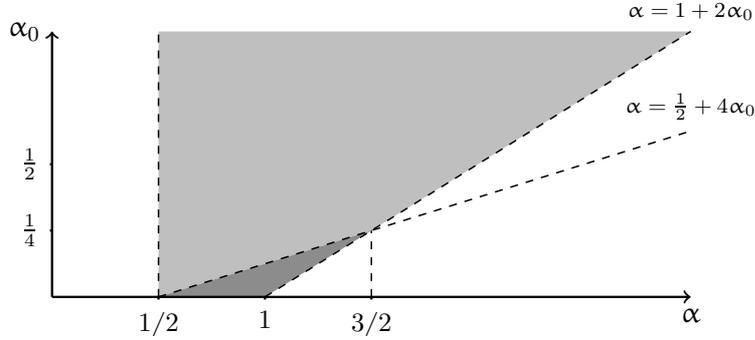
\begin{lemma}\label{l:Runiq}
Under the same assumptions of Theorem~\ref{t:Rexistuniq}, problem
\eqref{e:strongRv} has a unique solution $v\in C([0,\infty);V_\alpha)$.
\end{lemma}
\begin{proof}
Let $v_1$ and $v_2$ be two solutions of \eqref{e:strongRv} starting at
the same point and set $u_1=v_1+z$, $u_2=v_2+z$ and $w=v_1-v_2$. The
new function $w$ solves the following equation with random coefficients,
\begin{multline*}
\dot w + \nu Aw
 =   \chi_R(\|u_1\|_\alpha) B(u_1, w)
   + \chi_R(\|u_2\|_\alpha) B(w, u_2) + {}\\
   + [\chi_R(\|u_2\|_\alpha) - \chi_R(\|u_1\|_\alpha)] B(u_1,u_2),
\end{multline*}
with $w(0)=0$.  First, it is elementary to verify that there is $c>0$
such that for $x,y\geq0$,
\begin{equation}\label{e:chiprop}
|\chi(x) - \chi(y)|(1+x)(1+y)\leq c|x-y|.
\end{equation}
If $\alpha\geq\tfrac34$, set $\beta=\alpha-\tfrac34$ and estimate $w$
in $V_\beta$. Lemma~\ref{l:Bnostro} (with $a=b=\alpha$ and $c=-\beta$),
the above inequality and interpolation of $V_\alpha$ between $V_\beta$
and $V_{\beta+1}$ yield
\begin{align}\label{e:uniqgt32}
  \frac{d}{dt}\|w\|_\beta^2 + 2\nu \|w\|_{\beta+1}^2
    & \leq  c|\chi_R(\|u_2\|_\alpha) - \chi_R(\|u_1\|_\alpha)| \|u_1\|_\alpha\|u_2\|_\alpha\|w\|_{1+\beta}\notag\\
    &\quad  + cR\|w\|_\alpha \|w\|_{1+\beta}\\
    &\leq \nu\|w\|_{\beta+1}^2 + c_R\|w\|_\beta^2.\notag
\end{align}
If on the other hand $\alpha<\tfrac34$, we estimate $w$ in $H$. Lemma
\ref{l:Bnostro} (with $a=\tfrac32-\alpha$, $b=\alpha$ and $c=0$) and
interpolation of $V_\alpha$ and $V_{3/2-\alpha}$ between $H$ and $V$ yield
\begin{align*}
  \frac{d}{dt}\|w\|_H^2 + 2\nu \|w\|_V^2
    & \leq  c|\chi_R(\|u_2\|_\alpha) - \chi_R(\|u_1\|_\alpha)| \|u_1\|_{\frac32-\alpha} \|u_2\|_\alpha\|w\|_V\\
    &\quad  + cR\|w\|_V \|w\|_{\frac32-\alpha}\\
    &\leq \nu\|w\|_V^2 + c_R\|w\|_H^2(1 + \|u_1\|_{\frac32-\alpha}^{\frac{2}{1-\alpha}}),
\end{align*}
where $\|u_1\|_{3/2-\alpha}^{\frac{2}{1-\alpha}}$ is integrable in time thanks
to~\eqref{e:mildbound} and the fact that $\alpha>\tfrac12$.
In both cases Gronwall's lemma implies that $w\equiv0$, since $w(0) = 0$.
\end{proof}
\subsection{An estimate of the \emph{blow-up} time}

We next study the distribution of the random time
$\tau_{\alpha,R}:\Ons\to[0,\infty)$, defined in~\eqref{e:blowuptime}.
We start with an estimate of the tails of the solution
$z$ to~\eqref{e:stokes}, whose proof is standard (see \cite{DapZab92} for
instance, a proof in the case $\beta=2$ is given in \cite{FlaRom07}).
\begin{lemma}\label{l:Ztails}
Assume {\nuno} from Assumption~\ref{a:mainass} and let $\beta<1+2\alpha_0$.
Then there are $a_0>0$ and $c_0>0$ (depending only on $\alpha_0$, $\beta$, and
$\nu$) such that for all $K\geq\frac12$ and $\epsilon>0$,
\[
  \Pb\bigl[\sup_{s\leq\epsilon}\|z(t)\|_\beta\geq K]
  \leq c_0\e^{-a_0\frac{K^2}\epsilon}.
\]
\end{lemma}
\begin{proposition}\label{p:blowup}
Assume {\nuno} from Assumption~\ref{a:mainass} and let $\alpha\in(\tfrac12,1+2\alpha_0)$,
with $\alpha\neq\tfrac32$. There exists $c'=c'(\alpha)>0$ such that if $R\geq1$,
$x\in V_\alpha$ with $\|x\|_\alpha\leq\tfrac{R}{3}$ and if
$T\leq c'R^{-4/((2\alpha-1)\wedge2)}$ then
\[
  \Bigl\{\sup_{[0,T]}|z(t)|_\alpha\leq\frac{R}{3}\Bigr\}
  \subset \bigl\{\tau_x^\besp{\alpha,R}\geq T\bigr\},
\]
where $z$ is the solution to~\eqref{e:stokes}. In particular,
\[
  \Pb_x^\besp{\alpha,R}[\tau_x^\besp{\alpha,R}\leq T]
  \leq c_0\e^{-a_0\frac{R^2}{9T}}.
\]

If $\alpha=\tfrac32$, then for every $\epsilon<1$ such that $\alpha+\epsilon<1+2\alpha_0$
there is $c_\epsilon>0$ such that the same holds true on the event
$\{\sup_{[0,T]}|z(t)|_{\alpha+\epsilon}\leq R/3\}$ for
$T\leq c_\epsilon R^{-2/(1-\epsilon)}$.
\end{proposition}
\begin{proof}
Fix $x\in V_\alpha$ with $|x|_\alpha\leq\tfrac{R}{3}$, let $z$ be the solution
to~\eqref{e:stokes} and set $v^\besp{\alpha,R} = u_x^\besp{\alpha,R} - z$.
Assume first $\alpha>\tfrac32$. If $\sup_{[0,T]}|z(t)|_\alpha\leq\tfrac{R}{3}$,
inequality~\eqref{e:gt32} implies that
$\|v^\besp{\alpha,R}(t)\|_\alpha^2\leq\tfrac19 R^2+cR^4T$ for $t\leq T$, hence
\[
  \|u_x^\besp{\alpha,R}(t)\|
   \leq \|z(t)\|_\alpha + \|v^\besp{\alpha,R}(t)\|_\alpha
   \leq \frac{R}{3} + R\sqrt{\tfrac19 + cR^2T}
   \leq R
\]
if $T\leq c'R^{-2}$, for a suitable $c'$.
If on the other hand $\alpha<\tfrac32$, inequality \eqref{e:mildalpha} (which
holds for the full range $\alpha\in(\tfrac12,\tfrac32)$) yields
$\|v^\besp{\alpha,R}\|_\alpha\leq \tfrac13 R + cR^2T^{\frac14(2\alpha-1)}$,
hence $\|u_x^\besp{\alpha,R}(t)\|_\alpha\leq R$ for $t\leq T$, if
$T\leq c'R^{-4/(2\alpha-1)}$ and $\sup_{[0,T]}|z(t)|_\alpha\leq\tfrac{R}{3}$.

Finally, if $\alpha=\tfrac32$, we choose $\epsilon>0$ as we had done for \eqref{e:equal32}
so that $\|v(t)\|_\alpha\leq c_\epsilon R^{(2-\epsilon)/(1-\epsilon)}\sqrt{T}$
for $t\leq T$ and hence $\|u_x^\besp{\alpha,R}(t)\|_\alpha\leq R$ for $t\leq T$
if $T\leq c_\epsilon'R^{-2/(1-\epsilon)}$ and
$\sup_{[0,T]}|z(t)|_{\alpha+\epsilon}\leq\tfrac{R}{3}$.
\end{proof}
\subsection{Inequalities}

\begin{lemma}\label{l:mildbound}
Given $x,y\in[0,1)$ and $\delta>0$, $\eta>0$, let
\[
a(t) =
\begin{cases}
t^x, 				&\qquad 0\leq t\leq\delta,\\
\delta^x\e^{-\eta(t-\delta)},	&\qquad t>\delta.
\end{cases}
\]
Then $a$ is continuous on $[0,\infty)$, $|a(t)|\leq\delta^x$ and for all $t\geq0$,
\[
a(t)\int_0^t(t-s)^{-y}a(s)^{-1}\de[s]
\leq B(1-x,1-y)\delta^{1-y} + \eta^{y-1}\Gamma(1-y),
\]
where $B$ and $\Gamma$ are, respectively, the Beta and the Gamma functions.
\end{lemma}
\begin{proof}
Denote by $A(t)$ the function in the statement of the lemma. If $t\leq\delta$,
by a change of variables,
\[
A(t)
 =   t^x\int_0^t(t-s)^{-y}s^{-x}\de[s]
 =   t^{1-y}B(1-x,1-y)
\leq \delta^{1-y}B(1-x,1-y),
\]
while if $t>\delta$,
\[
\begin{aligned}
A(t)
& = \delta^x\e^{-\eta(t-\delta)}\int_0^\delta(t-s)^{-y}s^{-x}\de[s]
    + \int_\delta^t(t-s)^{-y}\e^{-\eta(t-s)}\de[s]\\
&\leq \delta^{1-y}B(1-x,1-y)
    + \eta^{y-1}\Gamma(1-y),
\end{aligned}
\]
where the first term is non-increasing in $t\geq\delta$ and we have used a
change of variables in the second term.
\end{proof}
Finally, we prove a slight generalisation of \cite[Lemma D.2]{FlaRom08} (a range
of parameters is covered by \cite[Lemma 2.1]{Tem95} or \cite[Proposition 6.4]{ConFoi88}).
First we need the following two elementary estimates.
\begin{lemma}\label{l:convserie1}
Let $\alpha\in\R$, then there is a number $c=c(\alpha)$ such that
for all $k_0\geq1$,
  \[
  \sum_{\vk\in\Z^3:\ 0<|\vk|\leq k_0}|\vk|^\alpha\leq
    \begin{cases}
      ck_0^{(\alpha+3)\vee 0}&\qquad\alpha\neq -3,\\
      c\log(1+k_0)   &\qquad\alpha = -3.
    \end{cases}
  \]
\end{lemma}
\begin{lemma}\label{l:convserie2}
Let $\alpha,\beta,\gamma\in\R$ be such that $2(\alpha+\beta+\gamma)\geq3$ if
$\beta<\tfrac32$, $\alpha+\gamma>0$ if $\beta=\tfrac32$ and $\alpha+\gamma\geq0$
if $\beta>\tfrac32$. Then there is a number $c=c(\alpha,\beta,\gamma)$ such that
for every $\vl\in\Z^3$, with $|\vl|>1$,
\[
  \sum_{\vm:\ |\vl+\vm|>2|\vm|}\frac1{|\vl|^{2\alpha}|\vm|^{2\beta}|\vl+\vm|^{2\gamma}}\leq c.
\]
\end{lemma}
\begin{proof}
First, notice that $\{\vm:|\vl+\vm|>2|\vm|\}\subset\{\vm:|\vm|<|\vl|\}$ and
so $|\vl+\vm|\leq2|\vl|$. We prove that $\frac23|\vl|\leq|\vl+\vm|$ holds
as well. If $|\vm|\leq\frac13|\vl|$, then
$|\vl+\vm|\geq |\vl|-|\vm| \geq\frac23|\vl|$. If on the other hand
$|\vm|\geq\frac13|\vl|$, then $|\vl+\vm|>2|\vm|\geq\frac23|\vl|$.
The conclusion now follows using the previous lemma.
\end{proof}
\begin{lemma}\label{l:Bnostro}
Let $a,b,c\in\R$ be such that $a\geq(-c)\vee0$, $b\geq(-c)\vee0$ and
$2(a + b + c)\geq3$ (with a strict inequality if at least one of the three
numbers is equal to $3/2$). Then there is a number $c_B=c_B(a,b,c)$ such that
\[
  \scal{B(u,v), w}\leq c_B \|u\|_a \|v\|_b \|w\|_{c+1}.
\]
for all $u\in V_a$, $v\in V_b$ and $w\in V_{c+1}$.
\end{lemma}
\begin{proof}
We proceed as in the proof of \cite[Lemma D.2]{FlaRom08}. In terms of Fourier
series $u(x)=\sum u_\vk \e^{\im\vk\cdot x}$ and $v(x)=\sum v_\vk \e^{\im\vk\cdot x}$,
hence
\[
  B(u,v) = \im\sum_{\vk\neq0}\Bigl(\sum_{\vl+\vm=\vk}(\vk\cdot u_\vl)P_\vk v_\vm\Bigr)\e^{\im\vk\cdot x},
\]
where $P_\vk:\R^3\to\R^3$ is the projection onto $\{y\in\R^3: y\cdot\vk=0\}$.
Therefore,
\[
\begin{aligned}
  \scal{B(u,v),w}
    & =   \Im\Bigl(\sum_{\vk\neq0}\overline{w}_\vk\Bigl(\sum_{\vl+\vm=\vk}(\vk\cdot u_\vl)P_\vk v_\vm\Bigr)\Bigr)\\
    &\leq \|w\|_{c+1}\Bigl(\sum_{\vk\neq0}|\vk|^{-2c}\Bigl|\sum_{\vl+\vm=\vk}|u_\vl|\,|v_\vm|\Bigr|^2\Bigr)^\frac12
\end{aligned}
\]
Divide the sum of the right-hand side of the above formula in the three terms
\term{A}, \term{B} and \term{C}, corresponding to the inner sum extended
respectively to 
  \begin{gather*}
   A_\vk=\{\vl+\vm=\vk,\ |\vl|\geq\frac{|\vk|}2, |\vm|\geq\frac{|\vk|}2\},\\
   B_\vk=\{\vl+\vm=\vk,\ |\vm|<\frac{|\vk|}2\},
   \qquad
   C_\vk=\{\vl+\vm=\vk,\ |\vl|<\frac{|\vk|}2\}.
   \end{gather*}
Set, for brevity, $U_\vk = |\vk|^a|u_\vk|$ and $V_\vk = |\vk|^a|v_\vk|$.
We start with the estimate of \term{A}. Since by Young's and Cauchy--Schwartz'
inequalities,
\[
  \term{A}^2
  \leq  2\|v\|_b^2\sum_{\vk\neq0}|\vk|^{-2c}\Bigl(\sum_{\vl+\vm=\vk}\!|\vl|^{-2(a+b)}U_\vl^2\Bigr)
      + 2\|u\|_a^2\sum_{\vk\neq0}|\vk|^{-2c}\Bigl(\sum_{\vl+\vm=\vk}\!|\vm|^{-2(a+b)}V_\vm^2\Bigr),
\]
by exchanging the sums in $\vk$ and $\vl$ and using Lemma~\ref{l:convserie1}
(we only consider the first term, one can proceed similarly for the second),
\[
  \sum_{\vk\neq0}|\vk|^{-2c}\sum_{\vl+\vm=\vk}|\vl|^{-2(a+b)}U_\vl^2
  =    \sum_{\vl\neq0}|\vl|^{-2(a+b)}U_\vl^2\sum_{|\vk|\leq 2|\vl|}|\vk|^{-2c}
  \leq c\|u\|_a^2,
\]
and so $\term{A}\leq c\|u\|_a\|v\|_b$. We estimate \term{B} using Cauchy--Schwartz'
inequality, exchanging the sums and using Lemma~\ref{l:convserie2},
\[
  \begin{aligned}
    \term{B}^2
    &\leq \|v\|_b^2\sum_{\vk\neq0}|\vk|^{-2c}\sum_{B_\vk}|\vl|^{-2a}|\vm|^{-2b} U_\vl^2\\
    & =    \|v\|_b^2\sum_{\vl\neq0}|\vl|^{-2a} U_\vl^2 \sum_{\vm:|\vl+\vm|>2|\vm|}|\vl+\vm|^{-2c}|\vm|^{-2b}\\
    &\leq c\|u\|_a^2\|v\|_b^2.
  \end{aligned}
\]
Finally, the term \term{C} can be obtained from \term{B} by exchanging $u$ with $v$ and
$\vl$ with $\vm$.
\end{proof}
\begin{corollary}\label{c:Bnostro}
If $a$, $b\geq0$, then there is $c_B>0$ such that for all $u\in V_a$
and $v\in V_b$,
\begin{equation*}
  \|A^{\frac{\delta}{2}}B(u,v)\|_H\leq c_B \|u\|_a \|v\|_b,
\end{equation*}
where $\delta=(a\wedge b - (\tfrac32-a\vee b)_+ - 1)$ if $a\vee b\neq\tfrac32$,
and $\delta<(a\wedge b-1)$ if $a\vee b=\tfrac32$ or $a\vee b=0$.
\end{corollary}

\end{document}